\documentclass[a4paper,twopage,reqno,12pt]{amsart}
\usepackage[top=30mm,right=30mm,bottom=30mm,left=30mm]{geometry}

\usepackage{tabularx}
\usepackage{graphicx}
\usepackage{amsmath}
\usepackage{amssymb}
\usepackage{amsfonts}
\usepackage{amsthm}
\usepackage{amstext}
\usepackage{amsbsy}
\usepackage{amsopn}
\usepackage{amscd}
\usepackage{enumerate}
\usepackage{color}
\usepackage[colorlinks]{hyperref}
\usepackage[hyperpageref]{backref}
\usepackage{multirow}
\usepackage{longtable}
\usepackage{float}
\usepackage{lscape}
\usepackage{pdflscape}
\usepackage{url}

\newtheorem{theorem}{Theorem}[section]

\newtheorem{lemma}[theorem]{Lemma}

\theoremstyle{definition}

\newtheorem{example}[theorem]{Example}

\numberwithin{equation}{section}

\renewcommand{\leq}{\leqslant}
\renewcommand{\geq}{\geqslant}

\begin{document}
\title[On flag-transitive $2$-designs]{On flag-transitive $2$-$(k^{2},k,\lambda )$ designs with $\lambda \mid k$.}

\author[]{ Alessandro Montinaro and Eliana Francot}

\thanks{Corresponding author: Alessandro Montinaro}

\address{Alessandro Montinaro and Eliana Francot, Dipartimento di Matematica e Fisica “E. De Giorgi”, University of Salento, Lecce, Italy}
\email{alessandro.montinaro@unisalento.it}
\email{eliana.francot@unisalento.it}

\subjclass[MSC 2020:]{05B05; 05B25; 20B25}%
\date{\today}%

\begin{abstract}
It is shown that, apart from the smallest Ree group, a flag-transitive automorphism group $G$ of a $2$-$(k^{2},k,\lambda )$ design $\mathcal{D}$, with $\lambda \mid k$, is either an affine group or an
almost simple classical group. Moreover, when $G$ is the smallest Ree group, $\mathcal{D}$ is isomorphic either to the $2$-$(6^{2},6,2)$ design or to one of the three $2$-$(6^{2},6,6)$ designs constructed in this paper. All the four $%
2$-designs have the $36$ secants of a nondegenerate conic $\mathcal{C}$ of $%
PG_{2}(8)$ as a point set and $6$-sets of secants in a remarkable
configuration as a block set.




\end{abstract}

\maketitle

\section{Introduction and Main Result}

A $2$-$(v,k,\lambda )$ \emph{design} $\mathcal{D}$ is a pair $(\mathcal{P},%
\mathcal{B})$ with a set $\mathcal{P}$ of $v$ points and a set $\mathcal{B}$
of blocks such that each block is a $k$-subset of $\mathcal{P}$ and each two
distinct points are contained in $\lambda $ blocks. We say $\mathcal{D}$ is 
\emph{nontrivial} if $2<k<v$. All $2$-$%
(v,k,\lambda )$ designs in this paper are assumed to be nontrivial. An
automorphism of $\mathcal{D}$ is a permutation of the point set which
preserves the block set. The set of all automorphisms of $\mathcal{D}$ with
the composition of permutations forms a group, denoted by $\mathrm{Aut(%
\mathcal{D})}$. For a subgroup $G$ of $\mathrm{Aut(\mathcal{D})}$, $G$ is
said to be \emph{point-primitive} if $G$ acts primitively on $\mathcal{P}$,
and said to be \emph{point-imprimitive} otherwise. A \emph{flag} of D is a
pair $(x,B)$ where $x$ is a point and $B$ is a block containing $x$. If $%
G\leq \mathrm{Aut(\mathcal{D})}$ acts transitively on the set of flags of $%
\mathcal{D}$, then we say that $G$ is \emph{flag-transitive} and that $%
\mathcal{D}$ is a \emph{flag-transitive design}.\ 

The $2$-$(v,k,\lambda )$ designs $\mathcal{D}$ admitting a flag-transitive
automorphism group $G$ have been widely studied by several authors. In 1990,
a classification of those with $\lambda =1$ and $%
G\nleq A\Gamma L_{1}(q)$ was announced by Buekenhout, Delandtsheer, Doyen,
Kleidman, Liebeck and Saxl in \cite{BDDKLS} and proven in \cite{BDD}, \cite%
{Da}, \cite{De0}, \cite{De}, \cite{Kle}, \cite{LiebF} and \cite{Saxl}. Since
then a special attention was given to the case $\lambda >1$. A
classification of the flag-transitive $2$-designs with $\gcd (r,\lambda )=1$%
, $\lambda >1$ and $G\nleq A\Gamma L_{1}(q)$, where $r$ is the replication
number of $\mathcal{D}$, has been announced by Alavi, Biliotti, Daneshkakh,
Montinaro, Zhou and their collaborators in \cite{glob} and proven in \cite{A}%
, \cite{A1}, \cite{ABD0}, \cite{ABD}, \cite{BM}, \cite{BMR}, \cite{MBF}, 
\cite{TZ}, \cite{Zie}, \cite{ZD}, \cite{ZZ0}, \cite{ZZ1}, \cite{ZZ2}, \cite%
{ZW} and \cite{ZGZ}. Moreover, recently the flag-transitive $2$-designs with 
$\lambda =2$ have been investigated by Devillers, Liang, Praeger and Xia in 
\cite{DLPX}, where it is shown that apart from the two known symmetric $2$-$%
(16,6,2)$ designs, $G$ is primitive of affine or almost simple type.
Moreover, a classification is provided when the socle of $G$ is isomorphic
to $PSL_{n}(q)\trianglelefteq G$ and $n\geq 3$.

The present paper represents a further contribution to the study of the
flag-transitive $2$-designs. More precisely, the flag-transitive $2$-$%
(k^{2},k,\lambda )$ designs with $\lambda \mid k$ are investigated. The reason of studying such $2$-designs is that they represent a natural generalization of the affine planes in
terms of parameters, and also because, it is shown in \cite{Monty} that, the blocks of imprimitivity of a family of flag-transitive, point-imprimitive symmetric $2$-designs investigated in \cite{PZ} have the structure of the $2$-designs analyzed here. The following result is obtained:

\begin{theorem}
\label{main}Let $\mathcal{D}$ be a $2$-$(k^{2},k,\lambda )$, with $\lambda
\mid k$, admitting a flag-transitive automorphism group $G$. Then $G$ is
point primitive and one of the following holds:

\begin{enumerate}
\item $G$ is an affine group.

\item $G$ is an almost simple classical group.

\item $\mathcal{D}$ is isomorphic to the $2$-$(36,6,2)$ design constructed
in Example \ref{Ex1} and $^{2}G_{2}(3)^{\prime }\trianglelefteq G\leq $ $%
^{2}G_{2}(3)$.

\item $\mathcal{D}$ is isomorphic to one of the three $2$-$(36,6,6)$ designs
constructed in Example \ref{Ex1} and $G\cong $ $^{2}G_{3}(3)$.
\end{enumerate}
\end{theorem}

Actually, (3) and (4) are special cases of (2), since $^{2}G_{3}(3)\cong
P\Gamma L_{2}(8)$. It worth noting that the example in (3) is not contained
in \cite{DLPX} and hence it is presumably new. A complete
classification of (1) for $G\nleq A\Gamma L_{1}(q)$, and of (2) are
contained in \cite{Mo1} and \cite{Mo2} respectively.

\section{Preliminary Reductions}

We first collect some useful results on flag-transitive designs.

\begin{lemma}
Let $\mathcal{D}$ be a $2$-$(k^{2},k,\lambda )$ design and let $b$ be the
number of blocks of $\mathcal{D}$. Then the number of blocks containing each
point of $\mathcal{D}$ is a constant $r$ satisfying the following:

\begin{enumerate}
\item $r=\lambda (k+1)$;

\item $b=\lambda k(k+1)$;

\item $\left( r/\lambda \right) ^{2}>k^{2}$.
\end{enumerate}
\end{lemma}

\bigskip

\begin{lemma}
\label{PP}If $\mathcal{D}$ is a $2$-$(k^{2},k,\lambda )$ design, with $%
\lambda \mid k$, admitting a flag-transitive automorphism group $G$, then
the following hold:

\begin{enumerate}
\item $G$ acts point-primitively on $\mathcal{D}$.

\item If $x$ is any point of $\mathcal{D}$, then $G_{x}$ is a large subgroup
of $G$.

\item $\left\vert y^{G_{x}}\right\vert =(k+1)\left\vert B\cap
y^{G_{x}}\right\vert $ for any point $y$ of $\mathcal{D}$, with $y\neq x$,
and for any block $B$ of $\mathcal{D}$ incident with $x$. In particular, $k+1
$ divides the length of each point-$G_{x}$-orbit on $\mathcal{D}$ distinct
from $\left\{ x\right\} $.
\end{enumerate}
\end{lemma}

\begin{proof}
The assertion (1) follows from \cite{Demb}, 2.3.7.c, since $r=(k+1)\lambda
>(k-3)\lambda $.

The flag-transitivity of $G$ on $\mathcal{D}$ implies $\left\vert
G\right\vert =k^{2}\left\vert G_{x}\right\vert $, $\left\vert
G_{x}\right\vert =\lambda \left( k+1\right) \left\vert G_{x,B}\right\vert $
and hence $\left\vert G\right\vert <\left\vert G_{x}\right\vert ^{3}$, which
is the assertion (2).

Let $y$ be any point of $\mathcal{D}$, $y\neq x$, and $B$ be any block of $%
\mathcal{D}$ incident with $x$. Since $(y^{G_{x}},B^{G_{x}})$ is a tactical
configuration by \cite{Demb}, 1.2.6, it follows that $\left\vert
y^{G_{x}}\right\vert \lambda =r\left\vert B\cap y^{G_{x}}\right\vert $.
Hence $\left\vert y^{G_{x}}\right\vert =(k+1)\left\vert B\cap
y^{G_{x}}\right\vert $ as $r=(k+1)\lambda $. This proves (3).
\end{proof}

\bigskip

The group $G$ is point-primitive on $\mathcal{D}$ by Lemma \ref{PP}(1). The
O'Nan-Scott Theorem classifies primitive groups into five types: (i) Affine
type; (ii) Almost simple type; (iii) Simple diagonal type; (iv) Product
type; (v) Twisted wreath product type (see \cite{LPS} for details). Hence,
the first part of the paper is devoted to prove that only families (i) and
(ii) occur. The result is achieved by adapting the techniques developed in 
\cite{ZZ} to the $2$-designs investigated here.

\begin{lemma}
\label{NoSDT}$G$ is not of simple diagonal type.
\end{lemma}

\begin{proof}
The proof is essentially that of \cite{ZZ}, Propositions 3.1, but we use $%
r=(k+1)\lambda $ instead of $\lambda \geq (r,\lambda )^{2}$.

Assume that $G$ is of diagonal type. Then 
\[
G\leq W=\left\{ (a_{1},...,a_{m})\pi \mid a_{i}\in Aut(T),\pi \in
S_{m},a_{i}\equiv a_{j} \mod Inn(T)\text{ for all }i,j\right\} 
\]%
and there is $x\in \mathcal{P}$ such that%
\[
G_{x}\leq \left\{ (a,...,a)\pi \mid a\in Aut(T),\pi \in S_{m}\right\} \cong
Aut(T)\times S_{m}
\]%
and $M_{x}=D=\left\{ (a,...,a)\mid a\in Inn(T)\right\} $ is a diagonal
subgroup of $M\cong T^{m}$. Put $\Sigma =\left\{ T_{1},...,T_{m}\right\} $,
where $T_{i}$ is identified with $\left\{ (1,...,t,...,1)\pi \mid t\in
T\right\} $ in the $i$-th position. Then $G$ acts on $\Sigma $ by \cite{LPS}%
. Moreover, the set $\mathcal{P}$ can be identified with the set $M/D$ of
the cosets of $D$ in $M$ so that $x=D(1,...,1)$, $k^{2}=\left\vert T\right\vert
^{m-1}$, since $v=k^{2}$, and for $y=D(t_{1},...,t_{m})$, $\psi =(s_{1},...,s_{m})\in M$, $%
\sigma \in Aut(T)$, $\pi \in S_{m}$, we have the actions%
\[
y^{\psi }=D(t_{1}s_{1},...,t_{m}s_{m})\text{, }y^{\sigma }=D(t_{1}^{\sigma
},...,t_{m}^{\sigma })\text{ and }y^{\pi }=D(t_{1\pi ^{-1}},...,t_{m\pi
^{-1}})\text{.}
\]%
Since $M\trianglelefteq G$ and $G$ is primitive on $\mathcal{P}$, $M$ is
transitive on $\mathcal{P}$. Since $T_{1}\trianglelefteq M$, all $T_{1}$%
-orbits on $\mathcal{P}$ have the same length $c>1$. Let $\Gamma _{1}$ be
the $T_{1}$-orbit containing $x$. For any $t_{1}=(t,1...,1)\in T_{1}$, we
have $x^{t_{1}}=D(t,1...,1)$. So that%
\[
\Gamma _{1}=x^{T_{1}}=\left\{ D(t,1...,1):t\in T\right\} 
\]%
and $\left\vert \Gamma _{1}\right\vert =\left\vert x^{T_{1}}\right\vert =c$.
Similarly, we define $\left\vert \Gamma _{i}\right\vert =\left\vert
x^{T_{i}}\right\vert $ for $1\leq i\leq m$. Clearly, $\Gamma _{i}\cap \Gamma
_{j}=\left\{ x\right\} $ for $i\neq j$ provided that $m\geq 2$.

Chose a point-$G_{x}$-orbit $\Delta $ in $\mathcal{P}-\left\{ x\right\} $
such that $\left\vert \Delta \cap \Gamma _{1}\right\vert =d\neq 0$. Let $%
m_{1}=\left[ G_{x}:N_{G_{x}}(T_{1})\right] $. Since $G_{x}$ is isomorphic to
a subgroup of $Aut(T)\times S_{m}$, and $G^{\Sigma }$ acts transitively on $%
\Sigma $, it follows that $m_{1}\leq m$ and hence 
\[
\left\vert \Delta \right\vert \leq m_{1}d\leq m\left\vert T\right\vert \text{%
.} 
\]%
Then $k+1\leq \left\vert \Delta \right\vert \leq m\left\vert T\right\vert $
by Lemma \ref{PP}(3). Since $v=k^{2}=\left\vert T\right\vert ^{m-1}$, we
have $\left\vert T\right\vert ^{(m-1)/2}<m\left\vert T\right\vert $ and
hence $60^{m-3}\leq \left\vert T\right\vert ^{m-3}<m^{2}$. Therefore, $m\leq
3$.

Since $r\mid \left\vert G_{x}\right\vert $ and $G_{x}$ is isomorphic to a
subgroup of $Aut(T)\times S_{m}$, it follows that $(k+1)\lambda \mid
\left\vert T\right\vert \left\vert Out(T)\right\vert m!$. On the other hand, 
$k+1\mid \left\vert T\right\vert ^{m-1}-1$, as $r/\lambda $ divides $k^{2}-1$%
. Thus $k+1\mid \left\vert Out(T)\right\vert m!$ and hence $\left\vert
T\right\vert ^{m-1}=k^{2}<\left\vert Out(T)\right\vert ^{2}\left( m!\right)
^{2}$ with $m\leq 3$. At this point the final part of the proof of \cite{ZZ}%
, Propositions 3.1, can be applied to show that no cases occur.
\end{proof}

\begin{lemma}
\label{NoTWPT}$G$ is not of twisted wreath product type
\end{lemma}

\begin{proof}
We may apply the same argument of \cite{ZZ}, Propositions 3.2, to show that
there is a point-$G_{x}$-orbit $\Delta $ in $\mathcal{P}-\left\{ x\right\} $
such $\left\vert \Delta \right\vert \leq m_{1}d\leq m\left\vert T\right\vert 
$ (this is shown in \cite{ZZ}, Propositions 3.2, without using the
assumption $\lambda \geq (r,\lambda )^{2}$). Then $k+1\leq m\left\vert
T\right\vert $ by Lemma \ref{PP}(3). On the other hand, $k+1>\left\vert
T\right\vert ^{m/2}$, since $k^{2}=v=\left\vert T\right\vert ^{m}$. Then $%
\left\vert T\right\vert ^{m/2}<m\left\vert T\right\vert $ and hence $%
60^{m-2}\leq m$ and $m\leq 2$, whereas $m\geq 6$ by \cite{LPS}.
\end{proof}

\begin{theorem}
\label{AffAS}$G$ is either of affine type or of almost simple type.
\end{theorem}

\begin{proof}
The group $G$ is neither of simple diagonal type nor of twisted wreath
product type by Lemmas \ref{NoSDT} and \ref{NoTWPT} respectively. Thus, in
order to complete the proof, we need to rule out the case where $G$ has a
product action on $\mathcal{P}$. Suppose the contrary. Then there is a group 
$K$ with a primitive action (of almost simple or diagonal type) on a set $%
\Gamma $ of size $v_{0}\geq 5$, such that $\mathcal{P}=\Gamma ^{m}$ and $%
G\leq K^{m}\rtimes S_{m}$, where $m\geq 2$. Let $x=(\gamma ,...,\gamma )$
and $y=(\delta ,...,\gamma )$ with $\delta \neq \gamma $ and set $W=K^{m}$
and $H=W\rtimes S_{m}$. Then $W_{x}\cong K_{\gamma }^{m}$, $W_{x,y}\cong
K_{\gamma ,\delta }\times K_{\gamma }^{m-1}$, $H_{x}=W_{x}\rtimes S_{m}$ and 
$K_{\gamma ,\delta }\times \left( K_{\gamma }^{m-1}\rtimes S_{m-1}\right)
\leq H_{x,y}$. Suppose that $K$ has rank $s$ on $\Gamma $, $s\geq 2$. Then
we may choose $\delta $ such that $\left[ K_{\gamma }:K_{\gamma ,\delta }%
\right] \leq \frac{v_{0}-1}{s-1}$. Hence, 
\[
\left\vert x^{H}\right\vert =\frac{\left\vert K_{\gamma }\right\vert
^{m}\cdot m!}{\left\vert K_{\gamma ,\delta }\right\vert \left\vert K_{\gamma
}\right\vert ^{m-1}\cdot (m-1)!}=\left[ K_{\gamma }:K_{\gamma ,\delta }%
\right] m\leq \frac{v_{0}-1}{s-1}m\text{.}
\]%
and, as $x^{G}\subseteq x^{H}$, we get 
\[
v_{0}^{m/2}=v^{1/2}<k+1\leq \left\vert x^{G}\right\vert \leq \left\vert
x^{H}\right\vert \leq m\frac{v_{0}-1}{s-1}<mv_{0}\text{.}
\]%
Then $m=2,3$ and $v_{0}<9$, as $m\geq 2$. If $m=3$, then $k^{2}=v_{0}^{3}$
and hence $v_{0}=4$ and $s=3$, whereas $v_{0}\geq 5$. Thus $m=s=2$. It
follows that, $K$ acts $2$-transitively on $\Gamma $, and $H=K^{2}\rtimes
S_{2}$ has rank $3$ with $H_{x}$-orbits $1$, $2(k-1)$ and $(k-1)^{2}$. Since
each $H_{x}$-orbit is union $G_{x}$-orbit, and since each $G_{x}$-orbit on $%
\mathcal{P}-\left\{ x\right\} $ has length divisible by $k+1$ by Lemma \ref%
{PP}(3), we obtain $k+1\mid 2(k-1)$ and hence $k=v_{0}=3$. So, we again
reach a contradiction as $v_{0}\geq 5$.
\end{proof}

\section{Proof of Theorem \protect\ref{main}}

In this section $G$ is an almost simple group. Hence, $X\trianglelefteq G\leq \mathrm{%
Aut(}X\mathrm{)}$, where $X$ is a non abelian simple group. Moreover $X$,
the socle of $G$, is either sporadic, or alternating, or an exceptional
group of Lie type, or classical. We analyze the first three cases
separately. The sporadic one is ruled out simply by filtering the groups
listed in \cite{At} with respect to the constraints for $X$ to have a
transitive permutation representation of degree $k^{2}$, and when this
occurs the corresponding stabilizer of a point in $X$ to have the order
divisible by $\frac{k+1}{\gcd (k+1,\left\vert \mathrm{Out(}X\mathrm{)}%
\right\vert )}$ (see Lemma \ref{orbits}). The alternating case is settled as
follows. We show that $X_{x}$, the stabilizer in $X$ of a point $x$ of $%
\mathcal{D}$, is a large maximal subgroup of $X$. Hence $X_{x}$ is listed in
Theorem 2 of \cite{AB}. Then we combine some group theoretical arguments, in
particular those developed in \cite{De}, together with some numerical
properties of the binomial coefficients to exclude the case. Finally, when $G$ is an exceptional
group of Lie type, the reduction to $^{2}G_{2}(3)$ in its permutation
representation of degree $36$ is settled by transferring the arguments
developed in \cite{A} and in \cite{ABD} to our context. The key point of the
analysis of the $2$-designs admitting $^{2}G_{2}(3)$ as a flag transitive
automorphism group is to see that $^{2}G_{2}(3)$ acts on $PG_{2}(8)$
preserving a nondegenerate conic $\mathcal{C}$, since $^{2}G_{2}(3)\cong $ $%
P\Gamma L_{2}(8)$. Hence, its permutation representation of degree $36$ is
equivalent to that on the set of secants to $\mathcal{C}$. Some geometry of $%
PG_{2}(8)$ is then used to complete the proof of the case.

\bigskip

\begin{lemma}
\label{orbits}Let $\mathcal{D}$ be a $2$-$(k^{2},k,\lambda )$ design, with $%
\lambda \mid k$, admitting a flag-transitive automorphism group $G$. If $x$
is any point of $\mathcal{D}$, then $\frac{k+1}{\gcd (k+1,\left\vert \mathrm{%
Out(}X\mathrm{)}\right\vert )}$ divides $\left\vert X_{x}\right\vert $.
\end{lemma}

\begin{proof}
Let $x$ be any point of $\mathcal{D}$. If $y$ is a point of $\mathcal{D}$,
with $y\neq x$, then $\left\vert y^{X_{x}}\right\vert =\frac{\left\vert
B\cap y^{G_{x}}\right\vert \left( k+1\right) }{\mu }$, where $\mu \left\vert
y^{X_{x}}\right\vert =\left\vert y^{G_{x}}\right\vert $, by Lemma \ref{PP}%
(3), as $X_{x}\trianglelefteq G_{x}$. On the other hand, $\mu $ divides $%
\left\vert \mathrm{Out(}X\mathrm{)}\right\vert $, as $\mu =\frac{\left[
G_{x}:X_{x}\right] }{\left[ G_{x,y}:X_{x,y}\right] }$. Therefore $\frac{k+1}{%
\gcd (k+1,\left\vert \mathrm{Out(}X\mathrm{)}\right\vert )}$ divides $%
\left\vert y^{X_{x}}\right\vert $ and hence $\left\vert X_{x}\right\vert $.
\end{proof}

\begin{lemma}
\label{NoSporadic}$X$ is not sporadic.
\end{lemma}

\begin{proof}
Assume that $X$ is sporadic. Then $X$ is listed in \cite{At}%
.

Assume that $X\cong M_{i}$, where $i=11,12,22,23$ or $24$. Since $\left[
X:X_{x}\right] =k^{2}$, it follows from \cite{KL}, Table 5.1.C, that $%
\lambda ^{2}=2^{a_{1}}3^{a_{2}}$ for some $a_{1},a_{2}\geq 2$. Then $k=12$
and either $X\cong M_{11}$ and $X_{x}\cong F_{55}$, or $X\cong M_{12}$ and $%
X_{x}\cong PSL_{2}(11)$ by \cite{At}. However, these cases are ruled out by
Lemma \ref{orbits}, since $\frac{k+1}{\gcd (k+1,\left\vert \mathrm{Out(}X%
\mathrm{)}\right\vert )}=13$\ does not divide $\left\vert X_{x}\right\vert $.

Assume that $X\cong J_{i}$, where $i=1,2,3$ or $4$. Then $k^{2}$ divides $%
2^{2}$, $2^{6}3^{2}5^{2}$, $2^{6}3^{4}$, or $2^{20}3^{2}11^{2}$,
respectively, by \cite{KL}, Table 5.1.C. Then $i=2$ and either $k=10$ and $%
X_{x}\cong PSU_{3}(3)$, or $k=60$ and $X_{x}\cong PSL_{2}(7)$ by \cite{At}.
However, these cases are ruled out as they contradict Lemma \ref{orbits}.

Assume that $X$ is isomorphic to one of the groups $HS$ or $McL$. By \cite%
{KL}, Table 5.1.C, $k^{2}$ divides $2^{8}3^{2}5^{2}$ or $2^{6}3^{6}5^{2}$
respectively. Then either $X\cong HS$, $X_{x}\cong M_{22}$ and $k=10$, or $%
X\cong McL$, $X_{x}\cong M_{22}$ and $k=45$. The latter is ruled out by
Lemma \ref{orbits}, since $\frac{k+1}{\gcd (k+1,\left\vert \mathrm{Out(}X%
\mathrm{)}\right\vert )}=23$\ does not divide $\left\vert X_{x}\right\vert $%
. The former yields $r=11\lambda $, where $\lambda =1,2,5$ or $10$ as $%
\lambda \mid k$. If $B$ is any block incident with $x$, then $\left[
X_{x}:X_{x,B}\right] $ divides $r$. Then $PSL_{3}(4)\trianglelefteq
G_{x,B}\leq P\Sigma L_{3}(4)$, and hence $\lambda =2$, by \cite{At}. Thus, $%
\left\vert G_{B}\right\vert =10\left\vert G_{x,B}\right\vert $, since $G_{B}$
is transitive on $B$, and hence $b=44$ or $88$. However, $HS\trianglelefteq
G\leq HS.Z_{2}$ has no such transitive representation degrees by \cite{At}.

It is straightforward to check that the remaining cases are ruled out
similarly, as they do not have transitive permutation representations of
degree $k^{2}$ by \cite{At} and \cite{W}.
\end{proof}

\begin{lemma}
\label{AltCases1o2}If $X\cong A_{n}$, then $n\neq 6$ and $G=X$. Moreover,
one of the following holds:

\begin{enumerate}
\item $X_{x}=\left( S_{t}\times S_{n-t}\right) \cap A_{n}$ where $1\leq
t<n/2 $.

\item $X_{x}=\left( S_{t}\wr S_{h}\right) \cap A_{n}$ where $n=th$ and $%
2\leq t\leq n/2$.
\end{enumerate}
\end{lemma}

\begin{proof}
Assume that $X\cong A_{n}$. If $n=6$, then $k^{2}=3^{2}$ or $6^{2}$, and the
former is ruled out by \cite{At}, whereas the latter yields $X_{x}\cong
D_{10}$. However this case cannot occur by Lemma \ref{orbits}, since $\frac{%
k+1}{\gcd (k+1,\left\vert \mathrm{Out(}X\mathrm{)}\right\vert )}=7$ does not
divide $\left\vert X_{x}\right\vert $. Thus $n\neq 6$, and hence $\left\vert 
\mathrm{Out(}X\mathrm{)}\right\vert =2$ by \cite{KL}, Theorem 5.1.3.

Let $\mu =\left[ G_{x}:X_{x}\right] $. Since $G=G_{x}X$, it follows that $%
G_{x}/X_{x}\cong G/X\leq \mathrm{Out(}X\mathrm{)}$ and hence $\mu \leq 2$.
Assume that $\mu =2$. Let $M$ be a maximal subgroup of $X$ containing $X_{x}$%
. Then $x^{M}$ is a block of imprimitivity for $X$ and hence $\left\vert
x^{M}\right\vert \mid k^{2}$. Since $x^{M}-\left\{ x\right\} $ is union of $%
X_{x}$-orbit, and each $X_{x}$-orbit distinct from $\left\{ x\right\} $ is
of length divisible by $\frac{k+1}{\gcd (k+1,2)}$ Lemma \ref{orbits}, it
follows that $\frac{k+1}{\gcd (k+1,2)}\mid \left\vert x^{M}\right\vert -1$.
Then $\left\vert x^{M}\right\vert =c\frac{k+1}{\gcd (k+1,2)}+1$, for some $%
c\geq 1$, and hence $k^{2}=d\left( c\frac{k+1}{\gcd (k+1,2)}+1\right) $ for
some $d\geq 1$. Thus%
\begin{equation}
dc\frac{k+1}{\gcd (k+1,2)}+d-1=k^{2}-1  \label{altrel}
\end{equation}%
and hence 
\begin{equation}
d=\theta \frac{k+1}{\gcd (k+1,2)}+1  \label{altre2}
\end{equation}%
for some $\theta \geq 1$. Now, substituting (\ref{altre2}) in (\ref{altrel})
we obtain 
\[
\theta c\left( k+1\right) +\theta \gcd (k+1,2)<\left( k-1\right) \left( \gcd
(k+1,2)\right) ^{2}
\]%
and hence $\theta c<\left( \gcd (k+1,2)\right) ^{2}$. Therefore, $k$ is odd
and $\left( \theta ,c\right) =(1,1),(1,2),(1,3)$ or $(3,1)$, which,
substituted in (\ref{altrel}) and (\ref{altre2}), yield $(k,d,c)=(15,9,3)$
or $(3,3,1)$. Then $n\leq \left[ A_{n}:M\right] =9,3$, respectively. Actually, $n=9$ by \cite{At}, as $n\geq 5$ and $n\neq 6$. However, $A_{9}$ has
no transitive permutation representations of degree $15^{2}$. Thus $\mu =1$, 
$G=X$, and hence $X_{x}$ is a large, maximal subgroup of $X$ by Lemma \ref%
{PP}(2). The last part of Lemma's statement follows from \cite{AB}, Theorem
2, since $\left[ X:X_{x}\right] =k^{2}$.
\end{proof}

\begin{lemma}
\label{Nemasimmetricno}Case (2) of Lemma \ref{AltCases1o2} cannot occur.
\end{lemma}

\begin{proof}
Assume that $X_{x}=\left( S_{t}\wr S_{h}\right) \cap A_{n}$, where $n=th$
and $2\leq t\leq n/2$. Then 
\begin{equation}
k^{2}=\left[ X:X_{x}\right] =\frac{th!}{\left( t!\right) ^{h}h!}\text{.}
\label{kvadrat}
\end{equation}%
By \cite{De}, there is a $G_{x}$-orbit (namely a $2$-cycle) of length either $%
h(h-1)$ or $t^{2}\frac{h(h-1)}{2}$ according to whether $t=2$ or $t>2$
respectively. Then $k+1$ divides the length of such a orbit by Lemma \ref{PP}%
(3), as $G=X$ by Lemma \ref{AltCases1o2}. Thus, in both cases we have $\frac{%
th!}{\left( t!\right) ^{h}h!}<\left( th\right) ^{2}$. The inequality%
\begin{equation}
\frac{h^{ht}}{\left( ht\right) ^{h}\cdot h}=\frac{e\left( \frac{th}{e}%
\right) ^{th}}{e^{h}t^{h}\left( \frac{t}{e}\right) ^{th}\cdot eh\left( \frac{%
h}{e}\right) ^{h}}\leq \frac{th!}{\left( t!\right) ^{h}h!}<\left( th\right)
^{2}  \label{sestougau}
\end{equation}%
is determined by using the known bound $e\left( \frac{f}{e}\right) ^{f}\leq
f!\leq ef\left( \frac{f}{e}\right) ^{n}$ for $f\in \mathbb{N}$, where $e$ is
the Napier's constant. Thus $h^{ht}<\left( ht\right) ^{h+2}\cdot h<\left(
ht\right) ^{h+3}$.

Assume that $h^{t}\geq (ht)^{2}$, then $\left( ht\right) ^{2h}\leq
h^{ht}<\left( ht\right) ^{h+3}$ and hence $h=2$, as $2\leq t\leq n/2$. Then (%
\ref{sestougau}) becomes $2^{t}\leq \binom{2t}{t}=\frac{2t!}{\left(
t!\right) ^{2}}<8t^{2}$ and so $t\leq 9$. However, (\ref{kvadrat}) is not
fulfilled for $h=2$ and any of these values of $t$.

Assume that $h^{t}<(ht)^{2}$. Then $2^{t-2}\leq h^{t-2}<t^{2}$ and hence
either $t=2$, or $3\leq t\leq 8$ and $h\leq 9$. Actually, the pairs $(h,t)$
in latter case do not fulfill (\ref{kvadrat}). Hence $t=2$. Since $h!\geq
2^{h}$, being $h\geq 2$, and (\ref{sestougau}) yields $2^{h}\leq \binom{2h}{h%
}\leq \frac{2h!}{2^{h}h!}<\left( 2h\right) ^{2}$ and hence $h\leq 8$.
However, (\ref{kvadrat}) is not fulfilled for $t=2$ and any of these values
of $h$.\ 
\end{proof}

\begin{lemma}
\label{nemanazimenicno}$X$ is not isomorphic to $A_{n}$.
\end{lemma}

\begin{proof}
In order to prove the assertion we need to rule out case (1) of Lemma \ref%
{AltCases1o2}, since case (2) has been ruled out in Lemma \ref%
{Nemasimmetricno}. Hence, assume that $X_{x}=\left( S_{t}\times
S_{n-t}\right) \cap A_{n}$ where $1\leq t<n/2$. Then the action on the point
set of $\mathcal{D}$ and on the $t$-subsets of $\left\{ 1,...,n\right\} $
are equivalent. Thus $k^{2}=\binom{n}{t}$. Then either $t\leq 2$, or $n=50$
and $t=3$ by \cite{AZ}, Chapter 3, since $1\leq t<n/2$.

Assume that $t=1$. Then $k^{2}=n\geq 9$, $X_{x}\cong A_{n-1}$ and hence $X$
acts point-$2$-transitively on $\mathcal{D}$. Then $X_{x,B}$ is a subgroup
of $X_{x}$ of index $r=\lambda (\sqrt{n}+1)$, where $\lambda \mid \sqrt{n}$.
Since $X_{x}$ has no subgroups of index less than $n-1$, as $n\geq 9$, it
follows that $\lambda =\sqrt{n}$. Then $r<\binom{n-1}{2}$ and hence $X_{x,B}$
is one of the subgroups of $A_{n-1}$ listed in \cite{DM}, Theorem 5.2.A, as $%
n\geq 9$. Assume that $X_{x,B}$ preserves an $s$-subset of $\left\{
1,...,n-1\right\} $, where $s=1,2$. Then $\binom{n-1}{s}=\sqrt{n}(\sqrt{n}%
+1) $, a contradiction. Then $n$ is odd and $r=\left[ X_{x}:X_{x,B}\right] =%
\frac{1}{2}\binom{n-1}{n/2}$, since $G=X$ by Lemma \ref{AltCases1o2}.
Therefore%
\[
2^{n/2-1}\leq \frac{1}{2}\binom{n-1}{n/2}=\sqrt{n}(\sqrt{n}+1)<2n 
\]%
and hence $n=9$, $X_{x}\cong A_{8}$ and $X_{x,B}\cong AGL_{3}(2)$, since $n$
is a square and $n$ is odd. However, this case is ruled out since $r\neq 15$.

Assume that $t=2$. Then $G$ has rank $3$ and, if $x$ is any point of $%
\mathcal{D}$, the $G_{x}$-orbits, say $\mathcal{O}_{i}$, $i=1,2,3$, have
length $1,\allowbreak 2n-4$ and $\binom{n-2}{2}$, respectively (see \cite{De}%
). Then $k+1$ divides the length of each of such orbits by Lemma \ref{PP}(3), as $G=X
$ by Lemma \ref{AltCases1o2}. Then $2\left( n-2\right) =c(k+1)$ for some $%
c\geq 1$ and hence%
\begin{equation}\label{trst}
\left( \frac{2\left( n-2\right) }{c}-1\right) ^{2}=\frac{n(n-1)}{2}\text{.}
\end{equation}
Then (\ref{trst}) yields $c=2$, $n=9$, $k=6$ and $X_{x}\cong S_{7}$. Note that, $\lambda
>1$ since there are no affine planes of order $6$. Also $\lambda \neq 6$,
since $S_{7}$ has no transitive permutation representations of degree $63$
by \cite{At}. Thus, either $\lambda =2$, $X_{x,B}\cong A_{6}$ and $X_{B}\cong
\left( A_{6}\times Z_{3}\right) :Z_{2}$, or $\lambda =3$, $X_{x,B}\cong
Z_{2}\times S_{5}$ and $X_{B}\cong S_{4}\times S_{5}$ again by \cite{At},
since $\lambda \mid k$. Then the actions on the point-set and on the
block-set of $\mathcal{D}$ are equivalent to the actions on the sets of $2$%
-subsets and $(\lambda +1)$-subsets of $\mathbb{N}_{9}=\left\{
1,...,9\right\} $ respectively. Then we may identify the point-set and the
block-set of $\mathcal{D}$ with these sets, respectively, in a way that the
incidence relation is the set-theoretic inclusion, as $X_{x,B}$ is
isomorphic either to $A_{6}$ or to $Z_{2}\times S_{5}$ according to whether $%
\lambda =2$ or $3$ correspondingly. So, $k\leq \lambda +1\leq 4$ and we reach a contradiction as $k=6$.

Assume that $t=3$ and $n=50$. Then $X_{x}\cong \left( S_{3}\times
S_{47}\right) \cap A_{50}$, $k=140$ and $r=141\lambda $, where $\lambda \mid
140$ and $b=141\cdot 140\cdot \lambda $. Then $b\leq 141\cdot 140^{2}<\binom{%
50}{6}$ and $b\neq \frac{1}{2}\binom{50}{25}$ and hence $A_{50-\ell }\leq
X_{B}\leq \left( S_{\ell }\times S_{50-\ell }\right) \cap A_{50}$, where $%
\ell <6$, by \cite{DM}, Theorem 5.2.A. Moreover $\left\vert X_{B}\right\vert 
$ is coprime to $47$, as $b=141\cdot 140\cdot \lambda $, and hence $\ell =4,5
$. Thus $f\binom{50}{\ell }=141\cdot 140\cdot \lambda $, where $f$ is the
index of $X_{B}$ in $\left( S_{\ell }\times S_{50-\ell }\right) \cap A_{50}$%
. If $\ell =5$, then $23$ divides $\binom{50}{5}$ and hence $\lambda $,
whereas $\lambda \mid 140$. Therefore, $\ell =4$, $A_{46}\leq X_{B}\leq
\left( D_{8}\times S_{46}\right) \cap A_{50}$ and hence $f=3\mu $ and $%
\lambda =35\mu $ with $\mu \leq 4$. Then $X_{B}$ preserves a $4$-set $Y$ of 
$\mathbb{N}_{50}=\left\{ 1,...,50\right\} $, whereas $X_{x}$ preserves a $3
$-set $Z$ of $\mathbb{N}_{50}$. Then $X_{x,B}$ preserves $Z\cup Y$. Set $%
w=\left\vert Z\cup Y\right\vert $, then $4\leq w\leq 7$, and $\binom{50}{w}$
must divide $\left[ X:X_{x,B}\right] $, which is equal to $\binom{50}{3}%
\cdot 141\cdot 35\cdot \mu $. Thus $w=4$ and hence $Z\subseteq Y$. Then $%
A_{46}\leq X_{x,B}$, and so $k\mid 8$, as $A_{46}\leq X_{B}\leq \left(
D_{8}\times S_{46}\right) \cap A_{50}$. This is a contradiction, as $k=6$.
\end{proof}

\begin{lemma}
\label{IzuzzetnoSimple}If $X$ is isomorphic to socle a finite exceptional
group of Lie type, then $\mathcal{D}$ is a $2$-$(36,6,\lambda )$ design,
where $\lambda \mid 6$ and $\lambda >1$, and $X$ is isomorphic either to $%
G_{2}(2)^{\prime }$ or to $^{2}G_{2}(3)^{\prime }$.
\end{lemma}

\begin{proof}
Recall that an exceptional group of Lie type is simple apart from $%
^{2}B_{2}(2)$, $G_{2}(2)$, $^{2}G_{2}(3)$, or $^{2}F_{4}(2)$ by\ \cite{KL},
Theorem 5.1.1. Thus, either $X$ is isomorphic to an exceptional simple group
of Lie type, or $X$ is isomorphic to one of the groups $^{2}B_{2}(2)$, $%
G_{2}(2)$, $^{2}G_{2}(3)$, or $^{2}F_{4}(2)$. If the latter occurs, since $G$
has a primitive permutation representation of degree $k^{2}$, then the
unique admissible cases to be analyzed are either $G_{2}(2)$ and $k=6$, or $%
^{2}G_{2}(3)$ and $k=3,6$ by \cite{At}. Suppose that $k=3$. Then $\lambda
=1,3$ as $\lambda \mid k$. If $\lambda =1$, then $\mathcal{D}\cong AG_{2}(3)$
and hence $G\leq AGL_{2}(3)$, which is impossible. Then $\lambda =3$, $r=12$
and hence $G\cong $ $P\Gamma L_{2}(8)$ and $G_{x}\cong F_{56}:Z_{3}$. So $%
\left\vert G_{x,B}\right\vert =14$ and $G_{x,B}\leq F_{56}$, a
contradiction. Thus $k=6$ and hence $\lambda =1,2,3$ or $6$ as $\lambda \mid
k$. Also, $\lambda >1$ since there are no affine planes of order $6$.
Therefore, $k=6$, $\lambda \mid 6$ and $\lambda >1$, and $X$ is isomorphic
either to $G_{2}(2)^{\prime }$ or to $^{2}G_{2}(3)^{\prime }$. 

Assume that $X$ is isomorphic to an exceptional simple group of Lie type.
Suppose that $G_{x}$ is not parabolic. Then $G_{x}$ is one of the groups
listed in \cite{ABD}, Theorem 1.6, or equivalently in Tables 2 and 3 of \cite%
{A}, since $G_{x}$ is a large maximal subgroup of $G$ by Lemma \ref{PP}(2).
In \cite{A}, Alavi points out that the first and the second column of Table
2 contains $X$ and $X_{x}$, respectively, the third one contains a lower
bound $\ell _{v}$ for $v=\left[ X:X_{x}\right] $, and the fourth one
contains an upper bound $u_{r}$ for $r$, determined by using the fact that $r
$ is a common divisor of $v-1$, of $\left\vert G_{x}\right\vert $ and of the
subdegrees of $G_{x}$ by his Lemma 4. Then, the author shows that $%
u_{r}^{2}<\ell _{v}$ for each case contained in Table 2, hence $r^{2}<v$,
and so all the cases in Table 2 are ruled out in his paper.

Our aim is to transfer Alavi's argument in order to rule out the possibility
for $G_{x}$ to not be a parabolic subgroup of $G$. Clearly, the first three
columns of Table 2 have the same meaning as in our paper, where $v=k^{2}$
for us. The role of $r$ and of Lemma 4 of \cite{A} are played by $r/\lambda
=k+1$ and by our Lemma \ref{PP}(3) respectively. Thus the upper bound $u_{r}$
for $r$ in \cite{A} becomes an upper bound for $r/\lambda $ in our context.
Therefore the inequality $u_{r}^{2}<\ell _{v}$ implies $\left( r/\lambda
\right) ^{2}<k^{2}$ but this is impossible in our context, since $r/\lambda
=k+1$. Thus all the groups listed in Tables 2 of \cite{A} cannot occur.

It is even easier to rule out the groups listed in Tables 3 of \cite{A} as
they are filtered with respect to the property that $v=k^{2}$. Indeed, we
obtain the following admissible cases:

\begin{enumerate}
\item $X\cong G_{2}(3)$, $X_{x}\cong 2^{3}:PSL_{3}(2)$ and $%
k^{2}=3528=\allowbreak 2^{3}3^{2}7^{2}$;

\item $X\cong G_{2}(4)$, $X_{x}\cong PSL_{2}(13)$ and $k^{2}=230400=%
\allowbreak 2^{10}3^{2}5^{2}$.
\end{enumerate}

Then $k+1$ is $43$ or $481$ respectively, but none of these divides the
order of the corresponding $G_{x}$. So, these cases violate Lemma \ref{PP}%
(3) and hence they are ruled out.

Assume that $G_{x}$ is a maximal parabolic subgroup of $G$. Assume that $%
E_{6}(q)$ is not contained in $G$. Then $G$ has a subdegree of order $p^{t}$
(e.g. see \cite{A}, Lemma 3, or \cite{Saxl}, Lemma 2.6). Then $\frac{r}{%
\lambda }\mid p^{t}$ and so $k+1=p^{s}$ for some $s\leq t$. Then $k=p^{s}-1$
and hence $k^{2}=(p^{s}-1)^{2}$. Then $s\leq \zeta _{p}(G)$, where $\zeta
_{p}(G)$ is defined in \cite{KL} (5.2.4), and is determined in Proposition
5.2.17.(i) and Table 5.2.C. If $s=\zeta _{p}(G)$, then $\left( p^{\zeta
_{p}(G)}-1\right) ^{2}\mid \left\vert X\right\vert $. On the other hand, $%
\left\vert X\right\vert $ is listed in \cite{KL}, Table 5.1.B, and hence
none of these groups admits $\left( p^{\zeta _{p}(G)}-1\right) ^{2}$ as a
divisor. Then $s<\zeta _{p}(G)$. Then $G_{x}$ contains a Sylow $u$-subgroup $%
G$, where $u$ is a primitive prime divisor of $p^{\zeta _{p}(G)}-1$, since $%
(p,\zeta _{p}(G))\neq (2,6)$ being $X\ncong G_{2}(2)^{\prime }$. On the
other hand, $G_{x}$ can be obtained by deleting the $i$-th node in the
Dynkin diagram of $X$, and we see that none of these groups is of order
divisible by $u$. Indeed, for instance, if $F_{4}(q)\trianglelefteq G$, $%
q=p^{f}$, then $\zeta _{p}(G)=12f$ and hence $G_{x}$ contains a Sylow $u$%
-subgroup of $G$, where $u$ is a primitive prime divisor of $p^{12f}-1$,
whereas the maximal parabolic subgroups are of type $B_{3}(p^{f})$, $%
C_{3}(p^{f})$ or $A_{1}(p^{f})\times A_{2}(p^{f})$ and none of these is
divisible by $u$. As stressed out in \cite{A}, Remark 1, even in the case $%
E_{6}(q)$, when $G$ contains a graph automorphism or $G_{x}$ is parabolic of
type $1,2$ or $4$, then $G$ has a subdegree of order $p^{t}$, and hence
these cases are excluded by the above argument. For the remaining maximal
parabolic subgroups we may use the same argument as \cite{A} at
pp.1012--1013, with $r/\lambda $ and Lemma \ref{PP}(3) in the role of $r$
and Lemma 4 of \cite{A}, respectively,\ to see that $\left( r/\lambda
\right) ^{2}>v=k^{2}$ is violated. Hence $E_{6}(q)\trianglelefteq G$ cannot
occur and the proof is thus completed.
\end{proof}

\begin{lemma}
\label{Izuzetno}$X$ is not isomorphic to $G_{2}(2)^{\prime }$.
\end{lemma}

\begin{proof}
Suppose that $\mathcal{D}$ is a $2$-$(36,6,\lambda )$ design, where $\lambda
\mid 6$ and $\lambda >1$, admitting a flag-transitive automorphism group $%
X\trianglelefteq G\leq \mathrm{Aut(}X\mathrm{)}$, where $X\cong
G_{2}(2)^{\prime }$. Therefore $PSL_{2}(7)\trianglelefteq G_{x}\leq
PGL_{2}(7)$ by \cite{At}, and $\lambda =2,3$ or $6$.

Assume that $\lambda =2$. Then $\left\vert G_{x,B}\right\vert =12$ and hence
either $G=X$ and $G_{x,B}\cong A_{4}$, or $G\neq X$ and $G_{x,B}\cong A_{4}$%
, $D_{12}$ by \cite{At}. Then $72$ divides $\left\vert G_{B}\right\vert $
and hence $G_{B}\leq \left( U:Z_{8}\right) .Z_{2}$, where $U$ is a Sylow $3$%
-subgroup of $G$. Then $U:Z_{8}\leq G_{B}$ since $36$ divides $\left\vert
G_{B}\cap \left( U:Z_{8}\right) \right\vert $ and since $U/Z(U):Z_{8}$ is a
Frobenius group (for instance, see \cite{Hup}, Satz II.10.12, since $%
G_{2}(2)^{\prime }\cong PSU_{3}(3)$). So $9$ divides $\left\vert
G_{x,B}\right\vert $, a contradiction.

Assume that $\lambda =3$ or $6$. Then $4\mid \left\vert X_{x,B}\right\vert $
as $r=21$ or $63$. If $G\neq X$, then the set of points of $\mathcal{D}$
distinct from $x$ is partitioned into two $G_{x}$-orbits, say $\mathcal{O}%
_{1}$, $\mathcal{O}_{2}$, of length $14$ and $21$ respectively by \cite%
{AtMod}. If $G=X$, then $\mathcal{O}_{1}$ is split into two $X_{x}$-orbits $%
\mathcal{O}_{1j}$, $j=1,2$, each of length $7$, whereas $\mathcal{O}_{2}$ is
also a $X_{x}$-orbit. Hence, $X_{x,y}\cong D_{8}$ for $y\in \mathcal{O}_{2}$.

Let $B$ be any block incident with $x$. Then $\left\vert \mathcal{O}%
_{i}\right\vert =7\left\vert \mathcal{O}_{i}\cap B\right\vert $, where $%
i=1,2 $, by Lemma \ref{PP}(3) and hence $\left\vert \mathcal{O}_{1}\cap
B\right\vert =2$ and $\left\vert \mathcal{O}_{2}\cap B\right\vert =3$. Then $%
X_{x,B}$ fixes $\mathcal{O}_{1}\cap B$ pointwise, since $\left\vert \mathcal{%
O}_{1j}\cap B\right\vert =1$ for each $j=1,2$. Then there is a non trivial
subgroup $W$ of $X_{x,B}$ of index at most $2$, such that $B\subseteq 
\mathrm{Fix(}W\mathrm{)}$, since $4\mid \left\vert X_{x,B}\right\vert $.

If $\lambda =3$, then $r=21$, $b=126$ and $X_{x,B}\cong D_{8}$. Hence, $W$
is isomorphic to $Z_{4},E_{4}$ or to $D_{8}$, and $N_{X_{x}}(W)$ is
isomorphic to $D_{8}$, $S_{4}$ or to $D_{8}$ respectively. The number of
points in $\mathcal{O}_{2}$ fixed by $W$ is given by the well known formula $%
\left\vert N_{X_{x}}(W)\right\vert \left\vert X_{x,y}\cap W^{X}\right\vert
/\left\vert X_{x,y}\right\vert $. Thus, it is easy to see that, $W$ fixes $6$
or $1$ points in $\mathcal{O}_{2}$ according to whether $W$ is or is not
isomorphic to $E_{4}$ respectively, since $X_{x,y}\cong D_{8}$. Therefore, $W\cong
E_{4}$, since $B\subseteq \mathrm{Fix(}W\mathrm{)}$ and $\left\vert \mathcal{%
O}_{2}\cap B\right\vert =3$. Then $W$ lies in the kernel $N$ of the action
of $G_{B}$ on $B$. Also $\mathrm{Fix(}W\mathrm{)}\neq \mathrm{Fix(}X_{x,B}%
\mathrm{)}$ as $X_{x,B}\cong D_{8}$ fixes exactly one point on $\mathcal{O}%
_{2}$, and hence $N\cap X_{B}=W$. Therefore $W\trianglelefteq G_{B}$ and
hence $G_{B}\leq N_{G}(W)$, where $N_{G}(W)$ is isomorphic either to $\left(
Z_{4}\right) ^{2}.S_{3}$ or to $\left( Z_{4}\right) ^{2}.D_{12}$ according
as $G=X$ or $G\neq X$, respectively, by \cite{At}. Note that, $\left[
N_{G}(W):G_{B}\right] =2$, since $\left[ G:N_{G}(W)\right] =63$ and $b=126$.
Thus, $G_{B}\vartriangleleft N_{G}(W)$ and $3$ divides $\left\vert
N_{X_{x}}(W)\cap G_{B}\right\vert $ as $N_{X_{x}}(W)\cong S_{4}$. Hence, $%
G_{x,B}$ contains a Sylow $3$-subgroup of $N_{G}(W)$, since these are cyclic
of order $3$, but this contradicts $[G_{B}:G_{x,B}]=6$.

If $\lambda =6$, then $r=42$, $b=252$ and hence either $X_{x,B}\cong Z_{4}$ or $%
X_{x,B}\cong E_{4}$. Therefore, $W$ is isomorphic to $Z_{2}$, $E_{4}$ or to $%
Z_{4}$. Actually, $W\ncong Z_{4}$ since $Z_{4}$ fixes one point on $\mathcal{%
O}_{2}\cap B$, whereas $B\subseteq \mathrm{Fix(}W\mathrm{)}$ and $\left\vert 
\mathcal{O}_{2}\cap B\right\vert =3$. 

Assume that $W\cong E_{4}$. Then $W=X_{x,B}$ and hence $W=X_{B}\cap N$,
where $N$ is defined above. Therefore $W\trianglelefteq G_{B}$ and hence $%
G_{B}\leq N_{G}(W)$. Moreover, $\left[ N_{G}(W):G_{B}\right] =4$, since $%
\left[ G:N_{G}(W)\right] =63$ and $b=252$. So, $G\cong G_{2}(2)$, $%
N_{G}(W)\cong \left( Z_{4}\right) ^{2}.D_{12}$ and $G_{B}\cong \left(
Z_{4}\right) ^{2}.Z_{3}$ by \cite{At}. Then $G_{B}\vartriangleleft N_{G}(W)$%
, $3$ divides $\left\vert N_{X_{x}}(W)\cap G_{B}\right\vert $ and we reach a
contradiction as above. 

Assume $W\cong Z_{2}$, $\left\vert \mathrm{Fix(}W\mathrm{)}\cap \mathcal{O}%
_{2}\right\vert =5$ and $B\subset \mathrm{Fix(}W\mathrm{)}$. As above, $W$
lies in the kernel $N$ of the action of $G_{B}$ on $B$. Then $\mathrm{Fix(}W%
\mathrm{)}\neq \mathrm{Fix(}X_{x,B}\mathrm{)}$. Indeed, it is true for $%
X_{x,B}\cong Z_{4}$, as any $Z_{4}$ fixes exactly one point on $\mathcal{O}%
_{2}$, and it is still true for $X_{x,B}\cong E_{4}$ as it follows from the
above arguments where it is shown that a subgroup isomorphic to $E_{4}$ cannot fix $B$
pointwise. Consequently, we have that $N\cap X_{B}=W$. Thus $G_{B}\leq
C_{G}(W)$, where $C_{G}(W)$ is isomorphic either to $GU_{2}(3)$ or to $%
\Gamma U_{2}(3)$ according as $G=X$ or $G\neq X$, respectively, by \cite{At}%
. Moreover, $\left[ C_{G}(W):G_{B}\right] =4$, as $b=252$, and hence $%
\left\vert X_{B}\right\vert =24$. Therefore, $X_{B}\cong Z.S_{3}$ or $%
X_{B}\cong SL_{2}(3)$ since $C_{X}(W)\cong \left( Z\circ SL_{2}(3)\right)
.Z_{2}$, where $Z$ \ is the center of $GU_{2}(3)$. Assume the former occurs.
Since $\left[ X_{B}:X_{x,B}\right] $ divides $6$, it follows that $Z\leq
X_{x,B}$. Then $Z$ fixes $B$ pointwise, as $Z$ is normal in $G_{B}$, and
hence $Z\leq N\cap X_{B}=W$. This is impossible, since $W\cong Z_{2}$
whereas $Z\cong Z_{4}$. Thus $X_{B}\cong SL_{2}(3)$. Therefore, either $%
G\cong G_{2}(2)^{\prime }$ and hence $G_{B}\cong SL_{2}(3)$, or $G\cong
G_{2}(2)$ and $G_{B}\cong GL_{2}(3)$, $Z_{4}\circ SL_{2}(3)$, as $\left[
C_{G}(W):G_{B}\right] =4$. However, it is easily seen with the aid of \cite{GAP} that, no $2$-designs occur.
\end{proof}

\bigskip

Before analyzing the remaining case involving $G_{3}(3)$. Recall some useful
facts about the action of this group in the desarguesian plane of order $8$.

\bigskip

Let $G\cong G_{3}(3)$. It is well known that $G\cong $ $P\Gamma L_{2}(8)$
acts on $PG_{2}(8)$ preserving a regular hyperoval, namely a $10$-arc,\
consisting of a nondegenerate conic $\mathcal{C}$ and it nucleus $N$ (see 
\cite{Hirsch}, Section 8). Moreover, $G$ acts primitively on the set $%
\mathcal{S}$ of the $36$ secants to $\mathcal{C}$ by \cite{At}. If $\ell $
is any secant to $\mathcal{C}$ and $\left\{ P_{1},P_{2}\right\} =\ell \cap 
\mathcal{C}$, then $G_{\ell }\cong F_{42}$ and the set $\mathcal{S}(\ell )$
consisting of the $14$ secant lines to $\mathcal{C}$ incident with $P_{1}$
or $P_{2}$ is a $G_{\ell }$-orbit, as $G$ acts $3$-transitively on $\mathcal{%
C}$. The complementary set $\mathcal{S}-\mathcal{S}(\ell )$ consisting of $%
21 $ secants intersecting $\ell $ in a point different from $P_{1},P_{2}$ is
also a $G_{\ell }$-orbit.

By \cite{At}, $G$ contains two conjugacy classes of subgroups of order $3$,
one contained in $G^{\prime }$ the other in $G-G^{\prime }$. If $%
\left\langle \eta \right\rangle $ and $\left\langle \gamma \right\rangle $
are the representatives of such classes, then $C_{G}(\eta )\cong S_{3}$ and $%
C_{G}(\gamma )\cong Z_{3}\times S_{3}$. Moreover, $C_{G}(\eta )=C_{G^{\prime
}}(\eta )$ and $C_{G^{\prime }}(\gamma )\cong S_{3}$. We may choose $\eta $
and $\gamma $ to belong to the same Sylow $3$-subgroup of $G$ in a way that $%
C_{G}(\gamma )=\left\langle \gamma \right\rangle \times C_{G}(\eta )$. Note
that, $C_{G}(\eta )=\left\langle \eta ,\sigma \right\rangle $, where $\sigma 
$ is an involutory elation of $PG_{2}(8)$ with center $C$ not in $\mathcal{C}
$ and axis $a$ tangent to $\mathcal{C}$ (a detailed description of the
collineations of the desarguesian plane can be found in \cite{Hirsch} and in 
\cite{HP}).

Set $K=\left\langle \gamma ,\sigma \right\rangle $. Then $K\cong Z_{6}$ is a
self-normalizing subgroup of $G$ by \cite{At}. Moreover, $K$ fixes exactly
one point $F$ on $\mathcal{C}$, and $\left\langle \gamma \right\rangle $
fixes two further points switched by $\sigma $, say $W$ and $W^{\sigma }$.
Then $C=W^{\sigma }W\cap F_{0}N$, $a=F_{0}N$. The set $\left\{ F,W,W^{\sigma
}\right\} $ is a $C_{G}(\gamma )$-orbit on $\mathcal{C}$. The set $\mathcal{C%
}\mathbf{-}\left\{ F,W,W^{\sigma }\right\} $ is split into two $\left\langle
\gamma \right\rangle $-orbits, $\left\{ P,P^{\gamma },P^{\gamma
^{2}}\right\} $ and $\left\{ P^{\sigma },P^{\sigma \gamma },P^{\sigma \gamma
^{2}}\right\} $, and these are also $\left\langle \gamma ,\eta \right\rangle 
$-orbits. Moreover, $\left\{ P,P^{\gamma },P^{\gamma ^{2}},P^{\sigma
},P^{\sigma \gamma },P^{\sigma \gamma ^{2}}\right\} $ is both a $K$-orbit
and a $C_{G}(\gamma )$-orbit on $\mathcal{C}$.

\begin{lemma}
\label{Uredu}The following hold:

\begin{enumerate}
\item $\left( PP^{\gamma }\right) ^{C_{G}(\gamma )}$ is the unique $%
C_{G}(\gamma )$-orbit on $\mathcal{S}$ of length $6$.

\item $\left( PP^{\gamma }\right) ^{C_{G}(\eta )}$ and $\left( FP^{\gamma
^{i}}\right) ^{C_{G}(\eta )}$, where $i=0,1,2$, are the unique $C_{G}(\eta )$%
-orbits on $\mathcal{S}$ of length $6$.

\item $\left( P^{\gamma }P\right) ^{K}$, $\left( P^{\gamma }P^{\sigma
}\right) ^{K}$, $\left( PW\right) ^{K}$, $\left( P^{\sigma }W\right) ^{K}$,
and $\left( PF\right) ^{K}$ are the unique $K$-orbits on $\mathcal{S}$ of
length $6$.
\end{enumerate}

In particular, $\left( PP^{\gamma }\right) ^{C_{G}(\gamma )}=\left(
PP^{\gamma }\right) ^{C_{G}(\eta )}=\left( P^{\gamma }P\right) ^{K}$.
\end{lemma}

\begin{proof}
Since $\left\{ P,P^{\gamma },P^{\gamma ^{2}}\right\} $ and $\left\{
P^{\sigma },P^{\sigma \gamma },P^{\sigma \gamma ^{2}}\right\} $ are $%
\left\langle \gamma ,\eta \right\rangle $-orbits, it follows that $%
\left\vert \left( PP^{\gamma }\right) ^{C_{G}(\gamma )}\right\vert \mid 6$.
On the other hand $\left\vert \left( PP^{\gamma }\right) ^{K}\right\vert =6$
and $\left( PP^{\gamma }\right) ^{K}\subseteq \left( PP^{\gamma }\right)
^{C_{G}(\gamma )}$. Thus $\left( PP^{\gamma }\right) ^{C_{G}(\gamma )}$ is a 
$C_{G}(\gamma )$-orbit on $\mathcal{S}$ of length $6$.

Since $\left\{ F,W,W^{\sigma }\right\} $ is a $C_{G}(\gamma )$-orbit on $%
\mathcal{C}$. It follows that $\left\vert \left( FW\right) ^{C_{G}(\gamma
)}\right\vert =3$. Moreover, $\left\vert \left( FP\right) ^{K}\right\vert =6$
and hence $\left\vert \left( FP\right) ^{C_{G}(\gamma )}\right\vert =18$, as 
$\left\{ P,P^{\gamma },P^{\gamma ^{2}},P^{\sigma },P^{\sigma \gamma
},P^{\sigma \gamma ^{2}}\right\} $ is both a $K$-orbit and a $C_{G}(\gamma )$%
-orbit on $\mathcal{C}$.

Since $\sigma $ fixes $PP^{\sigma }$, it follows that $\left( PP^{\sigma
}\right) ^{C_{G}(\gamma )}$ is of odd length. Assume that $\left\vert \left(
PP^{\sigma }\right) ^{C_{G}(\gamma )}\right\vert =3$. Then there is an
element $\vartheta $ in $\left\langle \gamma ,\eta \right\rangle $
preserving $PP^{\sigma }$ and hence fixing both $P$ and $P^{\sigma }$. Then $%
\vartheta $ fixes $\mathcal{C}\mathbf{-}\left\{ F,W,W^{\sigma }\right\} $
pointwise, since $\left\langle \gamma ,\eta \right\rangle $ is an elementary
abelian group of order $9$ acting transitively both on $\left\{ P,P^{\gamma
},P^{\gamma ^{2}}\right\} $ and on $\left\{ P^{\sigma },P^{\sigma \gamma
},P^{\sigma \gamma ^{2}}\right\} $. However, this is impossible. Hence $%
\left\vert \left( PP^{\sigma }\right) ^{C_{G}(\gamma )}\right\vert =9$. As $%
\left( FW\right) ^{C_{G}(\gamma )}\cup \left( FP\right) ^{C_{G}(\gamma
)}\cup \left( PP^{\sigma }\right) ^{C_{G}(\gamma )}$ covers $\mathcal{C}%
-\left( PP^{\gamma }\right) ^{C_{G}(\gamma )}$, the assertion (1) follows.

Since $C_{G}(\eta )\trianglelefteq C_{G}(\gamma )$, each $C_{G}(\eta )$%
-orbit on $\mathcal{S}$ of length $6$ lies either in $\left( PP^{\gamma
}\right) ^{C_{G}(\gamma )}$ or in $\left( FP\right) ^{C_{G}(\gamma )}$.
Since $C_{G}(\eta )=C_{G^{\prime }}(\eta )$, and $\left\langle \eta
\right\rangle $ acts semiregularly on $\mathcal{C}$, it results that $%
\left\langle \eta \right\rangle $ does not fix secants to $\mathcal{C}$. On
the other hand, since $C_{G}(\gamma )_{PP^{\gamma }}\cong Z_{3}$, $%
C_{G}(\gamma )_{PP^{\gamma }}\cap G^{\prime }=1$ and $C_{G}(\gamma )_{FP}=1$%
, it follows that $\left( PP^{\gamma }\right) ^{C_{G}(\eta )}=\left(
PP^{\gamma }\right) ^{C_{G}(\gamma )}$ and that $\left( FP\right)
^{C_{G}(\gamma )}$ is split into three $C_{G}(\eta )$-orbits each of length $%
6$, namely, $\left( FP^{\gamma ^{i}}\right) ^{C_{G}(\eta )}$ where $i=0,1,2$.

Finally, it is easy to check that $W^{\sigma }W$, $WF$, $P^{\sigma
}P$, $PF$, $PW$, $P^{\sigma }W$, $P^{\gamma }P$ and $P^{\gamma }P^{\sigma }$
are representatives of all $K$-orbits on $\mathcal{S}$ and these have
lengths $1$, $2$, $3$, $6$, $6$, $6$, $6$, $6$ respectively. Thus (3) holds.
\end{proof}

\begin{example}
\label{Ex1}Let $\ell _{1}=PP^{\gamma }$, $\ell _{2}=P^{\gamma }P^{\sigma }$, 
$\ell _{3}=PW$ and $\ell _{4}=P^{\sigma }W$, and let $B_{1}=\ell
_{1}^{C_{G}(\gamma )}$ and $B_{i}=\ell _{i}^{K}$ for $i=2,3,4$. Then the
following hold:

\begin{enumerate}
\item $\mathcal{D}_{1}=(\mathcal{S},B_{1}^{G})$ is a $2$-$(36,6,2)$ design
admitting $G$ as the full flag-transitive automorphism group. Moreover, $%
G^{\prime }$ acts flag-transitively on $\mathcal{D}_{1}$.

\item $\mathcal{D}_{i}=(\mathcal{S},G_{i}^{G})$, where $i=2,3,4$, is a $2$-$%
(36,6,6)$ design admitting $G$ as the unique flag-transitive automorphism
group.

\item $\mathcal{D}_{2}$, $\mathcal{D}_{3}$ and $\mathcal{D}_{4}$, are
pairwise non isomorphic.
\end{enumerate}
\end{example}

\begin{proof}
Let $Z$ denote $C_{G}(\gamma )$ for $i=1$ and $K$ for $i>1$. Clearly, $Z\leq
G_{B_{i}}$. Assume that $Z\neq G_{B_{i}}$. Then $1\neq G_{\ell
_{i},B_{i}}\leq F_{42}$ and $G_{B_{i}}$ does not contain elements of order $7
$, as $\left\vert B_{i}\right\vert =6$ and as the unique element of $G$
fixing more than three points on $\mathcal{C}$ is the identity. Thus $%
G_{\ell _{i},B_{i}}\leq Z_{6}$. Suppose that $\left\vert G_{\ell
_{i},B_{i}}\right\vert $ is even. Then either $36\mid \left\vert
G_{B_{1}}\right\vert $, or $12\mid \left\vert G_{B_{i}}\right\vert $ for $i>1
$. In the former case we obtain $G=G_{B_{1}}$ by \cite{At}, but this is
impossible. Hence $i>1$ and either $G_{B_{i}}\cong A_{4}$ or $G_{B_{i}}$
contains a Sylow $2$-subgroup of $G$. Both are ruled out. Indeed, the former
cannot occur since $Z_{6}\cong K\leq G_{B_{i}}$ but $A_{4}$ does not contain
such groups. In the latter case $G_{B_{i}}$ contains an involution $\alpha $
fixing $\ell _{i}\cap B_{i}$ pointwise. However, this is impossible since
the involutions are elations of $PG_{2}(8)$ and their unique fixed point is
the tangency point of their axis, whereas $\ell _{i}$ is a secant to $\mathcal{%
C}$. Thus $G_{\ell _{i},B_{i}}\cong Z_{3}$ and hence $\left\vert
G_{B_{i}}\right\vert =18$. Therefore, $G_{B_{i}}=C_{G}(\gamma )$. This is clear
for $i=1$, whereas, for $i>1$ it follows from $Z_{6}\cong K\leq G_{B_{i}}$
and $K\cap G_{\ell _{i},B_{i}}=1$. However, $G_{B_{1}}=C_{G}(\gamma )$
contradicts the assumption. Thus $i>1$ and hence $B_{i}=B_{1}$ by Lemma \ref%
{Uredu}(1), but we still obtain a contradiction since $B_{1}$ is also a $K$%
-orbit by Lemma \ref{Uredu}(3). Thus $G_{B_{i}}=Z$ for each $i=1,2,3,4$.
Therefore, by \cite{Demb}, 1.2.6, $\mathcal{D}_{i}=(\mathcal{S},B_{i}^{G})$
is a flag-transitive tactical configuration with parameters $%
(v,b,k,r)=(36,84,6,14)$ or $(36,252,6,42)$ according as $i=1$ or $i=2,3,4$
respectively.

In order to prove that $\mathcal{D}_{i}$ is a $2$-design with $\lambda =2$
for $i=1$ and $\lambda =6$ for $i>1$, bearing in mind that $G$ acts
transitively on $\mathcal{S}$, and for any $\ell \in \mathcal{S}$ the group $%
G_{\ell }$ acts transitively both on $\mathcal{S}(\ell )$ and on $\mathcal{S}%
-\mathcal{S}(\ell )$, it is enough to prove that there are precisely $%
\lambda $ elements of $B^{G}$ containing $\ell _{i}$ and any $m_{i}\in
B_{i}-\left\{ \ell _{i}\right\} $, $i$ fixed.

\begin{enumerate}
\item[(i).] $\mathcal{D}_{1}=(\mathcal{S},B_{1}^{G})$\textbf{\ is a }$2$%
\textbf{-}$(36,6,2)$\textbf{\ design admitting }$G$\textbf{\ as the full
flag-transitive automorphism group of }$\mathrm{\mathcal{D}}_{1}$\textbf{.
Moreover, }$G^{\prime }$\textbf{\ acts flag-transitively on }$\mathrm{%
\mathcal{D}}_{1}$\textbf{.}
\end{enumerate}

Let $m_{1}\in B_{1}$. Assume that $\ell _{1}\cap m_{1}\in \mathcal{C}$. Then 
$m_{1}\in \mathcal{S}(\ell _{1})$ and hence $\left\vert m_{1}^{G_{\ell
_{1}}}\right\vert =14$. Clearly $\left\vert B_{1}^{G_{\ell _{1}}}\right\vert
=14$ as $G_{\ell _{1},B_{1}}\cong Z_{3}$. Moreover, $\left\vert B_{1}\cap
m_{1}^{G_{\ell _{1}}}\right\vert =2$ and hence $\left( m_{1}^{G_{\ell
_{1}}},B_{1}^{G_{\ell _{1}}}\right) $ is a tactical configuration with
parameters $(v_{1},b_{1},k_{1},r_{1})=(14,14,2,2)$ by \cite{Demb}, 1.2.6.
Hence, the number of secants in $B_{1}^{G_{\ell _{1}}}$ containing both $%
\ell _{1}$ and $m_{1}$ is $2$. Then the number of secants in $B_{1}^{G}$
containing both $\ell _{1}$ and $m_{1}$ is $2$, as $\mathcal{D}_{1}$ is a
flag-transitive tactical configuration.

If $\ell _{1}\cap m_{1}\notin \mathcal{C}$. Then $m_{1}\in \mathcal{S}-%
\mathcal{S}(\ell _{1})$ and hence $\left\vert m_{1}^{G_{\ell
_{1}}}\right\vert =21$. Moreover, $\left\vert B_{1}^{G_{\ell
_{1}}}\right\vert =14$ and $\left\vert B_{1}\cap m_{1}^{G_{\ell
_{1}}}\right\vert =3$. Indeed, $B_{1}\cap m_{1}^{G_{\ell
_{1}}}=m_{1}^{\left\langle \gamma \right\rangle }$ and hence $\left(
m_{1}^{G_{\ell _{1}}},B_{1}^{G_{\ell _{1}}}\right) $ is a tactical
configuration with parameters $(v_{1},b_{1},k_{1},r_{1})=(21,14,3,2)$.
Hence, the number of elements in $B^{G}$ containing both $\ell _{1}$ and $%
m_{1}$ is $2$, as $\mathcal{D}_{1}$ is a flag-transitive tactical
configuration. Thus, there are precisely $2$ elements of $B^{G}$ containing
both $\ell _{1}$ and $m_{1}$ regardless $\ell _{1}\cap m_{1}$ lies or does
not lie in $\mathcal{C}$. Therefore, $\mathcal{D}_{1}=(\mathcal{S},B_{1}^{G})
$ is a $2$-$(36,6,2)$ design admitting $G$ as a flag-transitive automorphism
group.

Note that, $\mathrm{Aut(\mathcal{D}}_{1}\mathrm{)}=G$ as a consequence of
the O'Nan-Scott Theorem (e.g. see \cite{DM}, Theorem 4.1A), since $%
v=2^{2}\cdot 3^{2}$, $G=\mathrm{Aut(}G\mathrm{)}$ and $G\leq $ $\mathrm{Aut(%
\mathcal{D}}_{1}\mathrm{)}$. Thus $G$ is the full flag-transitive
automorphism group of $\mathrm{\mathcal{D}}_{1}$.

Since $r=14$, $\left( G^{\prime }\right) _{\ell }\cong D_{14}$, and $\left(
G^{\prime }\right) _{\ell ,B_{1}}\leq G_{\ell ,B_{1}}\cong Z_{3}$, it
follows that $\left( G^{\prime }\right) _{\ell ,B_{1}}=1$. Therefore, $\left[
\left( G^{\prime }\right) _{\ell }:\left( G^{\prime }\right) _{\ell ,B_{1}}%
\right] =14$ and hence $G^{\prime }\cong PSL_{2}(8)$ acts flag-transitively
on $\mathrm{\mathcal{D}}_{1}$.

\begin{enumerate}
\item[(ii).] $\mathcal{D}_{2}=(\mathcal{S},B_{2}^{G})$\textbf{\ is a }$2$%
\textbf{-}$(36,6,6)$\textbf{\ design admitting }$G$\textbf{\ as the unique
flag-transitive automorphism group.}
\end{enumerate}

Let $m_{2}\in B_{2}$. Then $\left( m_{2}^{G_{\ell _{2}}},B_{2}^{G_{\ell
_{2}}}\right) $ is a tactical configuration with parameters $%
(v_{2},b_{2},k_{2},r_{2})$ equal either to $(14,42,2,6)$ or to $(21,42,3,6)$
according as $\ell _{2}\cap m_{2}$ lies or does not lie in $\mathcal{C}$
respectively. Therefore, $\mathcal{D}_{2}$ is a $2$-$(36,6,6)$ design admitting $G$ as a
flag-transitive automorphism group.

Arguing as in (i), we see that $G$ is the full flag-transitive automorphism
group of $\mathrm{\mathcal{D}}_{2}$. Let $H$ be the minimal flag-transitive
automorphism group of $\mathrm{\mathcal{D}}_{2}$. Then $H\leq G$ and hence $%
H=G$ by \cite{At}, since $2^{3}\cdot 3^{3}\cdot 7\mid \left[ G:G_{\ell
_{2},B_{2}}\right] $. Thus $G$ is the unique flag-transitive automorphism
group of $\mathcal{D}_{2}$.

\begin{enumerate}
\item[(iii).] $\mathcal{D}_{i}=(\mathcal{S},B_{i}^{G})$\textbf{, }$i=3,4$%
\textbf{, is a }$2$\textbf{-}$(36,6,6)$\textbf{\ design admitting }$G$%
\textbf{\ as the unique flag-transitive automorphism group.}
\end{enumerate}

Let $m_{3}\in B_{3}$ and suppose that $\ell _{3}\cap m_{3}\in \mathcal{C}$.
Since $G$ is $3$-transitive on $\mathcal{C}$, we may assume that $\ell _{3}\cap m_{3}=\left\{
P\right\} $. Then $G_{\ell _{3},m_{3}}$ fixes the vertices of the triangle
inscribed in $\mathcal{C}$ having $\ell _{3},m_{3}$ as two of its three
sides. Hence $G_{\ell _{3},m_{3}}\cong Z_{3}$ since $G$ is $3$-transitive on 
$\mathcal{C}$. Thus, $\left\vert m_{3}^{G_{\ell _{3}}}\right\vert =14$.
Since $G_{B_{1}}$ acts regularly on $B_{1}$ it follows that $\left\vert
B^{G_{\ell _{3}}}\right\vert =42$. Finally, $B$ contains exactly $3$ lines
of $m_{3}^{G_{\ell _{3}}}$ including $\ell _{3}$. Indeed, $G_{\ell _{3}}$
contains a cyclic group of order $7$ acting regularly on $\mathcal{C}-\ell
_{3}$. Thus $\left\vert B\cap m_{3}^{G_{\ell _{3}}}\right\vert =2$ and hence 
$\left( m_{3}^{G_{\ell _{3}}},B^{G_{\ell _{3}}}\right) $ is a tactical
configuration with parameters $(v_{3},b_{3},k_{3},r_{3})=(14,42,2,6)$. Thus
the number of blocks containing both $\ell _{3}$ and $m_{3}$ is $6$ as $%
\mathcal{D}_{3}$ is a flag-transitive tactical configuration.

Suppose that $\ell _{3}\cap m_{3}\notin \mathcal{C}$. We may assume that $%
W\in \ell _{3}$ and $W^{\sigma }\in m_{3}$. Then $G_{\ell _{3},m_{3}}$ is
generated by the unique elation of $PG_{2}(8)$ lying in $G$ and with center $%
\ell _{3}\cap m_{3}$. Thus $G_{\ell _{3},m_{3}}\cong Z_{2}$ and hence $%
\left\vert m_{3}^{G_{\ell _{3}}}\right\vert =21$. As above $\left\vert
B^{G_{\ell _{3}}}\right\vert =42$. Also $\left\vert B\cap m_{3}^{G_{\ell
_{3}}}\right\vert =3$. Indeed $G_{\ell _{3},W^{\sigma }}\cong Z_{6}$ acts
regularly on $\mathcal{C}-(\ell _{3}\cup \left\{ W^{\sigma }\right\} )$.
Therefore $\left( m_{3}^{G_{\ell _{3}}},B^{G_{\ell _{3}}}\right) $ is a
tactical configuration with parameters $(v_{3}^{\prime },b_{3}^{\prime
},k_{3}^{\prime },r_{3}^{\prime })=(21,42,3,6)$. Thus the number of blocks
containing both $\ell _{3}$ and $m_{3}$ is $6$ as $\mathcal{D}_{3}$ is a
flag-transitive tactical configuration. Therefore, $\mathcal{D}_{3}=(%
\mathcal{S},B_{2}^{G})$ is a $2$-$(36,6,6)$ design admitting $G$ as a
flag-transitive automorphism group. Arguing as in (i) and (ii), we see that $G
$ is the unique flag-transitive automorphism group of $\mathrm{\mathcal{D}}_{3}$%
. The statement (iii) for $\mathrm{\mathcal{D}}_{4}$ is proven similarly.

\begin{enumerate}
\item[(iv).] $\mathcal{D}_{2}$\textbf{, }$\mathcal{D}_{3}$\textbf{\ and }$%
\mathcal{D}_{4}$\textbf{, are pairwise non isomorphic.}
\end{enumerate}

Since $G_{B_{i}}=K$ is self-normalizing in $G$, where $i=1,2,3$, it follows
that $B_{i}$ is the unique block of $\mathcal{D}_{i}$ preserved by $G_{B_{i}}
$. Clearly,\ $\mathcal{D}_{2}\ncong \mathcal{D}_{3}$ and $\mathcal{D}%
_{2}\ncong \mathcal{D}_{4}$, since none of the $6$ secants lying in $B_{2}$
contains a point\ fixed by $\left\langle \gamma \right\rangle $, whereas $%
B_{3}$ and $B_{4}$ do.

Suppose that $\Phi $ is an isomorphism from $\mathcal{D}_{3}$ onto $\mathcal{%
D}_{4}$. Since $G$ acts point-transitively on $\mathcal{D}_{i}$, $i=3,4$, we
may assume that $\Phi $ fixes $F$. Also $G^{\Phi }=G$, since $G$ is the full
flag-transitive automorphism group of $\mathcal{D}_{i}$. Then $\left[ \Phi ,G%
\right] =1$, since $\mathrm{Aut(}G\mathrm{)}=G$, and hence $\Phi =1$ as $%
\Phi $ fixes $F$. Then $\mathcal{D}_{3}=\mathcal{D}_{4}$. Then there is $%
\delta \in G$ such that $B_{1}^{\delta }=B_{2}$. Hence $\delta \in
N_{G}(G_{B_{3}})$, as $G_{B_{3}}=G_{B_{4}}=K$. Then $\delta \in G_{B_{3}}$,
since $K$ is self-normalizing in $G$, and hence $B_{3}=B_{4}$, a
contradiction. Thus $\mathcal{D}_{2}\ncong \mathcal{D}_{3}$.
\end{proof}

\bigskip

\begin{theorem}
\label{compl}The following hold:

\begin{enumerate}
\item If $\mathcal{D}$ is a $2$-$(36^{2},6,\lambda )$ design, with $\lambda
\mid 6$, admitting $G$ as a flag-transitive automorphism group, then either $%
\lambda =2$ and $\mathcal{D}$ is isomorphic to $\mathcal{D}_{1}$, or $%
\lambda =6$ and $\mathcal{D}$ is isomorphic to one of the $\mathcal{D}_{i}$,
where $i=2$, $3$ or $4$.

\item If $\mathcal{D}$ is a $2$-$(36^{2},6,2)$ design admitting $G^{\prime }$ as a
flag-transitive automorphism group, then $\mathcal{D}$ is isomorphic to $%
\mathcal{D}_{1}$.
\end{enumerate}
\end{theorem}

\begin{proof}
Let $\mathcal{D}=(\mathcal{P},\mathcal{B})$ be a $2$-$(36^{2},6,\lambda )$,
with $\lambda \mid 6$, admitting $G$ a flag-transitive automorphism group.
Then $G$ acts point-primitively on $\mathcal{D}$ by Lemma \ref{PP}(1). On
the other hand, $G$ has a unique permutation representation of degree $36$
by \cite{At} and this one is equivalent to that on $\mathcal{S}$, the set of
secants to the nondegenerate conic $\mathcal{C}$ of $PG_{2}(8)$ preserved by 
$G$. Hence, we may identify the point set of $\mathcal{D}$ with $\mathcal{S}$%
. Thus, any block $B$ of $\mathcal{D}$ consists of $6$ secants to $\mathcal{C%
}$.

Assume that $\lambda =2$ and that $G^{\prime }$ acts flag-transitively on $%
\mathcal{D}$. Set $J=G^{\prime }$. Then $\left\vert J_{B}\right\vert =6$ and
hence $J_{B}$ is a $J$-conjugate of $C_{G}(\eta )$ (recall that $C_{G}(\eta
)=C_{J}(\eta )$). Without loss, we may assume that $J_{B}=C_{G}(\eta )$.
Then $B$ is one of the following $C_{G}(\eta )$-orbits $\left( PP^{\gamma
}\right) ^{C_{G}(\eta )}$ and $\left( FP^{\gamma ^{i}}\right) ^{C_{G}(\eta
)} $, where $i=0,1,2$, by Lemma \ref{Uredu}(3).

Assume that $B=\left( FP\right) ^{C_{G}(\eta )}$. Set $\ell =FP$. Since $%
C_{G}(\eta )$ acts transitively on $\left\{ F,W,W^{\sigma }\right\} $, it
follows that $B=\left\{ \ell ,\ell ^{\sigma },\ell ^{\eta },\ell ^{\sigma
\eta },\ell ^{\eta ^{2}},\ell ^{\sigma \eta ^{2}}\right\} $, where $\ell
\cap \ell ^{\sigma }=\left\{ F\right\} $, $\ell ^{\eta }\cap \ell ^{\sigma
\eta }=\left\{ W\right\} $ and $\ell ^{\eta ^{2}}\cap \ell ^{\sigma \eta
^{2}}=\left\{ W^{\sigma }\right\} $. Moreover, as $C_{G}(\eta )$ acts
regularly on $\mathcal{C}-\left\{ F,W,W^{\sigma }\right\} $, any two secants
lying in $B$ do not intersect in $\mathcal{C}-\left\{ F,W,W^{\sigma
}\right\} $. Thus, through any point of $\mathcal{C}-\left\{ F,W,W^{\sigma
}\right\} $ there is exactly one secant to $\mathcal{C}$ lying in $B$ and
incident with the point.

Let $m\in B$, $m\neq \ell $. Assume that $\ell \cap m\in \mathcal{C}$. Then $%
m=\ell ^{\sigma }\in \mathcal{S}(\ell )$ and hence $\left\vert m^{J_{\ell
}}\right\vert =12$ as $\mathcal{S}(\ell )$ is also a $J_{\ell }$-orbit,
being $J_{\ell }\cong D_{14}$. Clearly $\left\vert B^{J_{\ell }}\right\vert
=14$ as $J_{\ell ,B}=1$. If there is $e\in B\cap m^{J_{\ell }}$, with $e\neq
m$, then $e\in \mathcal{S}(\ell )$ and hence $e\cap \ell \in \mathcal{C}$.
Then $e\cap \ell \in \mathcal{C}-\left\{ F,W,W^{\sigma }\right\} $, as $%
e\neq m=\ell ^{\sigma }$, and we obtain a contradiction, since through any
point of $\mathcal{C}-\left\{ F,W,W^{\sigma }\right\} $ there is exactly one
secant to $\mathcal{C}$ lying in $B$ and incident with the point. Thus $%
\left\vert B\cap m^{J_{\ell }}\right\vert =1$ and hence $\left( m^{J_{\ell
}},B^{J_{\ell }}\right) $ is a tactical configuration with parameters $%
(v^{\prime },b^{\prime },k^{\prime },r^{\prime })=(14,14,1,1)$ by \cite{Demb}%
, 1.2.6. Hence the number of elements in $B^{J_{\ell }}$ containing both $%
\ell _{1}$ and $m_{1}$ is $1$. Then the number of $B^{J}$ containing both $%
\ell _{1}$ and $m_{1}$ is $1$, as $\mathcal{D}$ is flag-transitive by our
assumption. However, that is impossible as it contradicts the assumption $%
\lambda =2$. The cases $B=\left( FP^{\gamma ^{i}}\right) ^{C_{G}(\eta )}$,
with $i=1,2$, are similarly ruled out. Then $B=\left( PP^{\gamma }\right)
^{C_{G}(\eta )}$ and hence $B=\left( P^{\gamma }P\right) ^{C_{G}(\gamma )}$
by Lemma \ref{Uredu}. Thus, $\mathcal{D}\cong \mathcal{D}_{1}$ by Example \ref%
{Ex1}(1).

Assume that $\lambda =2$ and that $G$ acts flag-transitively on $\mathcal{D}$%
. Let $\ell \in B$, then $G_{\ell ,B}\cong Z_{3}$ and hence $\left\vert
G_{B}\right\vert =18$. Moreover, $G_{\ell ,B}\cap G^{\prime }=1$ and $%
G_{B}\cap G^{\prime }\cong S_{3}$, since $G^{\prime }\cong PSL_{2}(8)$, and
hence $G_{B}\cong Z_{3}\times S_{3}$. Since the centralizer in $G$ of any
subgroup of order $3$ of $G^{\prime }$ is of order $27$ by \cite{At}, it
follows that $G_{B}=C_{G}(\rho )$ for some element $\rho $ of order $3$
lying in $G-G^{\prime }$. We may assume that $G_{B}=C_{G}(\gamma )$, since
the subgroups of order $3$ of $G$ intersecting $G^{\prime }$ in $1$ lies in
one conjugacy class under $G$ again by \cite{At}. Thus, $B=B_{1}$ by Lemma %
\ref{Uredu}(1) and so $\mathcal{D}\cong \mathcal{D}_{1}$ by Example \ref{Ex1}%
(1).

Assume that $\lambda =3$. Then $G_{\ell ,B}\cong Z_{2}$ and $\left\vert
G_{B}\right\vert =12$, as $k=6$, and hence $G_{B}\cong A_{4}$ by \cite{At}.
Let $\alpha ,\beta ,\delta \in G_{B}$ such that $\left\langle \alpha ,\beta
\right\rangle \cong E_{4}$ and $o(\delta )=3$. Since $G_{B}$ preserves $%
\mathcal{C}$, the group $\left\langle \alpha ,\beta \right\rangle $ consists
of elations of $PG_{2}(8)$ having the same axis $a$ tangent to $\mathcal{C}$
and distinct centers $C_{\alpha }$, $C_{\beta }$ and $C_{\alpha \beta }$.
Furthermore, $\left\langle \delta \right\rangle $ fixes $a$ and permutes
transitively $\left\{ C_{\alpha },C_{\beta },C_{\alpha \beta }\right\} $.
Then the block $B$ consists of two secants incident with $C_{\alpha }$, two
ones incident with $C_{\beta }$ and two ones incident with $C_{\alpha \beta }
$. We may assume that $C_{\alpha }\in \ell $. Let $m\in B-\left\{
\ell \right\} $ such that $C_{\alpha }\in m$. Then $\left( m^{G_{\ell
}},B^{G_{\ell }}\right) $ is a tactical configuration with parameters $%
(v^{\prime \prime },b^{\prime \prime },k^{\prime \prime },r^{\prime \prime
})=(21,21,1,1)$ by \cite{Demb}, 1.2.6. Hence the number of elements in $%
B^{G_{\ell }}$ containing both $\ell $ and $m$ is $1$. Then the number of $%
B^{G}$ containing both $\ell $ and $m$ is $1$, as $\mathcal{D}$ is
flag-transitive. However, that is impossible as it contradicts the
assumption $\lambda =3$.

Assume that $\lambda =6$. Then $\left\vert G_{B}\right\vert =6$. If $%
G_{B}\leq G^{\prime }$, then $G_{B}\cong S_{3}$ and hence is a $G$-conjugate
of $C_{G}(\eta )$ by \cite{At}. Without loss, we may assume that $%
G_{B}=C_{G}(\eta )$. Then $B$ is one of the $C_{G}(\eta )$-orbits $\left(
PP^{\gamma }\right) ^{C_{G}(\eta )}$ and $\left( FP^{\gamma ^{t}}\right)
^{C_{G}(\eta )}$, where $t=0,1,2$, by Lemma \ref{Uredu}(2). If $B=\left(
PP^{\gamma }\right) ^{C_{G}(\eta )}$, then $B=\left( P^{\gamma }P\right)
^{C_{G}(\gamma )}$ again by Lemma \ref{Uredu}, and hence $\mathcal{D}\cong 
\mathcal{D}_{1}$ by Example \ref{Ex1}(1), whereas $\lambda =6$. So, $B\neq
\left( PP^{\gamma }\right) ^{C_{G}(\eta )}$ and hence $B_{t}=\left(
FP^{\gamma ^{t}}\right) ^{C_{G}(\eta )}$ for some $t=0,1,2$.

Let $t=0$ and let $m=\ell ^{\sigma }$. Then $m\cap \ell \in \mathcal{C}$. A
similar argument to that used for the case $\lambda =2$ shows that $\left(
m^{G_{\ell }},B^{G_{\ell }}\right) $ is a tactical configuration with
parameters $(v^{\prime \prime \prime },b^{\prime \prime \prime },k^{\prime
\prime \prime },r^{\prime \prime \prime })=(14,14,1,1)$ by \cite{Demb},
1.2.6, since through any point of $\mathcal{C}-\left\{ F,W,W^{\sigma
}\right\} $ there is exactly one secant to $\mathcal{C}$ lying in $B$ and
incident with the point. Hence the number of elements in $B^{G_{\ell }}$
containing both $\ell $ and $m$ is $1$. Then the number of $B^{G}$
containing both $\ell $ and $m$ is $1$, as $\mathcal{D}$ is flag-transitive.
So $\lambda =1$, but this contradicts our assumption. The cases $t=1,2$ are
excluded similarly.

Assume that $G_{B}\nleq G^{\prime }$. Then $G_{B}\cong Z_{6}$ and hence is a 
$G$-conjugate of $K$. Thus, without loss, we may assume that $G_{B}=K$. Then 
$B$ is one of the following $K$-orbits on $\mathcal{S}$: $\left( P^{\gamma
}P\right) ^{K}$, $\left( P^{\gamma }P^{\sigma }\right) ^{K}$, $\left(
PW\right) ^{K}$, $\left( P^{\sigma }W\right) ^{K}$, and $\left( PF\right)
^{K}$ by Lemma \ref{Uredu}(3). Then $B=\left( P^{\gamma }P\right)
^{K}=\left( P^{\gamma }P\right) ^{C_{G}(\gamma )}$ implies $\mathcal{D}\cong 
\mathcal{D}_{1}$ by Example \ref{Ex1}(1), and we again reach a contradiction
as $\lambda =6$. Thus, $B\neq \left( P^{\gamma }P\right) ^{K}$. Also $B\neq
\left( PF\right) ^{K}$, otherwise all the secants lying in any block of $%
\mathcal{D}$ intersects in a point, whereas for any two secants $s$ and $%
s^{\prime }$ to~$\mathcal{C}$ such that $s\cap s^{\prime }\notin \mathcal{C}$%
, then there are no blocks of $\mathcal{D}$ incident with them. Thus, $B$ is
either $\left( P^{\gamma }P^{\sigma }\right) ^{K}$ or $\left( PW\right) ^{K}$, or $\left( P^{\sigma }W\right) ^{K}$, and we obtain $\mathcal{D}\cong 
\mathcal{D}_{i}$, where $i=2$, $3$ or $4$, respectively, by Example \ref{Ex1}%
(2).
\end{proof}

\bigskip

\begin{proof}[Proof of Theorem \protect\ref{main}]
By Theorem \ref{AffAS}, $Soc(G)$, the socle of $G$, is either an elementary abelian $p$-group
for some prime $p$ or a non abelian simple group.

Assume that the latter occurs. Then $X$ is neither sporadic nor an
alternating group by Lemmas \ref{NoSporadic} and \ref{nemanazimenicno}
respectively. If $X$ is classical, then assertion (1) is immediate, but also
(2b)--(2c) follow by Theorem \ref{compl}, since $PSL_{2}(8)\cong $ $%
^{2}G_{2}(3)^{\prime }$. Finally, if $X$ is isomorphic to the socle a finite
exceptional group of Lie type, then $2$-$(36,6,\lambda )$ design, where $%
\lambda =2,3$ or $6$, and $X\cong $ $^{2}G_{2}(3)^{\prime }$ by Lemmas \ref%
{IzuzzetnoSimple} and \ref{Izuzetno}. Then the assertions (2b)--(2c) follow
again from Theorem \ref{compl}. This completes the proof.
\end{proof}

\end{document}